\documentclass{amsart}

\newtheorem{theorem}{Theorem}[section]
\newtheorem{lemma}[theorem]{Lemma}

\theoremstyle{definition}

\newtheorem{corollary}[theorem]{Corollary}
\newtheorem{proposition}[theorem]{Proposition}

\theoremstyle{remark}
\newtheorem{remark}[theorem]{Remark}

\numberwithin{equation}{section}

\usepackage{amssymb}

\newcommand{\Z}{\mathbb{Z}}
\newcommand{\N}{\mathbb{N}}
\newcommand{\R}{\mathbb{R}}
\newcommand{\eps}{\varepsilon}

\newcommand{\ov}{\overline}

\newcommand{\SR}{{\mathcal R}}

\newcommand{\SU}{{\mathcal U}}

\newcommand{\wt}{\widetilde}
\newcommand{\dm}{\diff^{1}_{m}(M)}
\newcommand{\dw}{\diff^{1}_{\omega}(M)}
\newcommand{\gm}{{\mathcal F}^1_m(M)}
\newcommand{\gw}{{\mathcal F}^{1}_{\omega}(M)}
\newcommand{\vf}{\mathfrak{X}^1(M)}
\newcommand{\vm}{\mathfrak{X}^1_m(M)}

\newcommand{\sm}{\mathfrak{X}^*_m(M)}

\newcommand{\diff}{\operatorname{Diff}}

\begin{document}

\title{Hyperbolicity in the Volume Preserving Scenario}

\author{Alexander Arbieto}
\address{Instituto de Matem\'atica,
Universidade Federal do Rio de Janeiro,
P. O. Box 68530,
21945-970 Rio de Janeiro, Brazil.}
\email{arbieto@im.ufrj.br}
\thanks{A.A. was partially supported by CNPq Grant and Faperj.}

\author{Thiago Catalan}
\address{Instituto de Ci\^encias, Matem\'atica e Computa\c{c}\~ao,
Universidade de São Paulo,
16-33739153 S\~ao Carlos-SP, Brazil}
\email{catalan@icmc.usp.br}
\thanks{T.C. was partially supported by Fapesp.}


\date{\today}


\begin{abstract}
Hayashi has extended a result of Ma\~n\'e, proving  that every diffeomorphism $f$ which has a $C^1$-neighborhood $\mathcal{U}$, where all periodic points of any $g\in\mathcal{U}$ are hyperbolic, it is an Axiom A diffeomorphism. Here, we prove the analogous result in the volume preserving  scenario.
\end{abstract}

\maketitle

\section{Introduction and Statement of the Results}

Let $M$ be a $C^{\infty}$ $d$-dimensional  Riemannian manifold without boundary and let
$\dm$ denote the set of diffeomorphisms which preserves the Lebesgue measure $m$ induced by the Riemannian metric. We endow this space with the $C^1$-topology.

In the theory of dynamical systems, one important question is to know whether robust dynamical properties in the phase space leads to differentiable properties of the system. For instance, one of the most important properties that a system can have is stability. This says that any system close enough to the initial have the same orbit structure of the initial. In other terms, this says that there is a topological conjugacy between this system and the initial one.

In a striking article \cite{M3} Ma\~n\'e proves that any $C^1$ structurally stable diffeomorphism is an Axiom A diffeomorphism. In \cite{P2} Palis extended this result to $\Omega$-stable diffeomorphisms. Actually, Ma\~n\'e believe that a weaker property than $\Omega$-stability should be enough to guarantee the Axiom A property. Let us elaborate on this property.

Given a diffeomorphism $f$ over $M$, a periodic point $p$ of $f$ is {\it hyperbolic} if $Df^{\tau(p)}$ has eigenvalues with absolute values different of one, where $\tau(p)$ is the period of $p$. In the space of $C^1$ diffeomorphisms over $M$, $\diff^1(M)$, we can define the set $\mathcal{F}^1(M)$ as the set of diffeomorphisms $f\in \diff^1(M)$ which have a $C^1$-neighborhood $\mathcal{U}\subset \diff^1(M)$ such that if $g\in \SU$ then any periodic point of $g$ is hyperbolic. In \cite{Hayashi}, Hayashi proved that any diffeomorphism in $\mathcal{F}^1(M)$ is Axiom A, which means that the periodic points are dense in the nonwandering set $\Omega(f)$ and the last one is a hyperbolic set. We recall that in dimension two, this was proved by Man\'e \cite{Mane}. This was also studied for flows without singularities by Gan and Wen in \cite{GW}.

Observe that in the volume preserving scenario, the Axiom A condition is equivalent to the diffeomorphism be Anosov, since  $\Omega(f)=M$ by Poincar\'e Recurrence Theorem. Hence, it is a natural question if Hayashi's and Man\'e's results still holds in the volume preserving scenario. Actually, it seems that the arguments of Ma\~n\'e holds in this case, specially the ones related to the perturbations. Moreover, using recent generic results in the volume preserving context much of the arguments of the original proof can be avoided. The purpose of this article is to do this. We want to stress out that we follow the same lines of Ma\~n\'e's original argument once we show that periodic points must have the same index. We observe also that Bessa and Rocha also have analogous results in \cite{BR} and together with Ferreira in \cite{BFR}, for the context of incompressible and Hamiltonian flows, but in lower dimensions (three and four respectively). We will also discuss what kind of results our arguments could prove in the contex of incompressible flows in any dimension.

We define the set $\gm$ as the set of diffeomorphisms $f\in \dm$ which have a $C^1$-neighborhood $\mathcal{U}\subset \dm$ such that if $g\in \SU$ then any periodic point of $g$ is hyperbolic.

If $f\in \diff^1(M)$ then an $f-$invariant compact set $\Lambda$ of $M$ is called a {\it hyperbolic set} if there is a continuous and
$Df$-invariant splitting $T_{\Lambda}M=E^s\oplus E^u$ such that there are constants $0<\lambda<1$ and $C>0$, satisfying
$$
\|Df_x^k| E^s(x)\|\leq C\lambda^k \quad \text{and}\quad \|Df_x^{-k}| E^u(x)\|\leq C\lambda^k,
$$
for every $x\in \Lambda$ and $n>0$. We say that $f$ is an {\it Anosov diffeomorphism},
if $M$ is a hyperbolic set for $f$. The main result of this article is the following,

\begin{theorem}
Any diffeomorphism in $\gm$ is Anosov.
\label{maintheorem}\end{theorem}

If the manifold is symplectic, i.e. it possess a non-degenerated closed 2-form $\omega$, then we can define the set $\gw$ using only symplectic diffeomorphisms, i.e. which preserves $\omega$, as we did using volume preserving diffeomorphisms. We denote by $\dw$ the space of symplectic diffeomorphisms. It was proved by Newhouse in \cite{Newhouse} that any element of $\gw$ is Anosov.

We observe that, since the neighborhoods of the diffeomorphisms are taken in the respectively spaces $\dm$,  $\dw$, or $\diff^1(M)$ we could take no relation between $\gm$, $\gw$ and $\mathcal{F}^1(M)$ direct from definition. But, as a corollary of the previous theorem we obtain,

\begin{corollary}
$\gw\subset \gm \subset \mathcal{F}^1(M)$.
\end{corollary}

Since the arguments to prove the main theorem involves heterodimensional cycles, it is natural to try to relate it with the well known Palis conjecture \cite{Palis} in the volume preserving scenario. In fact the arguments needed to do this are due to Crovisier and they were outlined in \cite{CROVISIER}. Since it is related to our main theorem, we write it here just for sake of completeness.

\begin{corollary}
If $f\in\dm$ is not an Anosov diffeomorphism then it can be approximated by one diffeomorphism, either exhibiting an heterodimensional cycle if the dimension of $M$ is greater than two or exhibiting an homoclinic tangency if the dimension of $M$ is two.
\label{palis conj}\end{corollary}

Actually, the statement for surfaces was proved by Newhouse in \cite{Newhouse}. Since for surfaces any volume preserving diffeomorphism is symplectic.

This paper is organized as follows. In section 2, we recall the Franks lemma in the volume preserving scenario,  which is one of the main tools used in the proof. In section 3, we prove that the index for hyperbolic periodic points is constant in a neighborhood of any $f\in \gm$. In section 4, we give a proof of our main theorem. In section 5, we recall Crovisier's arguments about Palis conjecture. In section 6, we point out how the arguments in the previous sections could be use to prove some analogous results of Gan and Wen \cite{GW}, Toyoshiba \cite{Toyo} and see how this extends a corollary of Bessa and Rocha, in \cite{BR} to higher dimensions. Finally, in the appendix we describe the adaptations to the volume preserving case of an argument by Ma\~n\'e.

\section{Franks-type Lemma}

One of the most useful and basic perturbation lemmas is the Franks lemma \cite{FRANKS}. This lemma enable us to perform non-linear perturbations along a finite piece of an orbit simply performing arguments from Linear Algebra.

In what follows, we will recall this lemma in the volume-preserving  scenario, which is contained in the work of Bonatti-Diaz-Pujals \cite{BDP}, proposition 7.4. See also \cite{LLS}.

\begin{lemma}
\label{l.franks}
Let $f\in \dm$ and $\SU$ be a $C^1$-neighborhood of $f$ in $\dm$. Then, there exist a neighborhood $\SU_0\subset \SU$ of $f$ and $\delta>0$ such that if $g\in \SU_0(f)$, $S=\{p_1,\dots,p_m\}\subset M$ and $\{L_i:T_{p_i}M\to T_{p_{i+1}}M\}_{i=1}^m$ are linear maps belonging to $SL(d)$ satisfying $\|L_i-Dg(p_i)\|\leq\delta$ for $i=1,\dots m$ then there exists $h\in\SU(f)$ such that $h(p_i)=g(p_i)$ and $Dh(p_i)=L_i$.
\end{lemma}

\section{Index of Periodic Orbits}

In this section we analyze the index of periodic orbits of diffeomorphisms in $\gm$. By definition the index of a hyperbolic periodic orbit is the dimension of its stable space. We will see that in the volume preserving case, the property of have two periodic hyperbolic saddles with different indices cannot happen if $f\in \gm$. The main result in this section is the following.

\begin{proposition}
\label{p.index}
Let $f\in \gm$ then there exist a neighborhood $\SU$ of $f$ in $\dm$ and an integer $i$ such that for every diffeomorphism $g\in\SU$ and every hyperbolic periodic orbit $p$ of $g$, the index of $p$ with respect to $g$ is $i$.
\end{proposition}

This can be seen through the creation of heterodimensional cycles, thus the proposition can be seen as the volume-preserving and discrete version of a result by Gan and Wen (see theorem 4.1 of \cite{GW}). But, we can also apply a volume-preserving version of a result by Abdenur-Bonatti-Crovisier-Diaz-Wen (corollary 2 of \cite{ABCDW}). In fact, if we define the Lyapunov exponent vector of a hyperbolic periodic point $p$ of period $\tau(p)$ by
$$v=(\frac{\log|\mu_1|}{\tau(p)},\dots,\frac{\log|\mu_d|}{\tau(p)}),$$
where $\mu_1,\dots,\mu_d$ are the eigenvalues of $Df^{\tau(p)}(p)$ ordered by their moduli. Then the volume preserving version of corollary 2 of \cite{ABCDW} would gives us a residual subset of $\dm$ such that for any homoclinic class of a diffeomorphism belonging to this residual subset, the closure of the set of Lyapunov vectors of the saddles of this homoclinic class is convex. This could be done, using the available perturbation tools in the volume preserving case, the Pasting lemma \cite{A} and the regularization theorem of \'Avila \cite{Avila}. Moreover, by a result of Bonatti and Crovisier \cite{BONATTICROVISIER}, we can assume that for any difffeomorphism in this residual subset the whole manifold is an homoclinic class. Thus, this would give us a saddle with an eigenvalue with norm close to one and by Franks lemma \ref{l.franks} this would create a non-hyperbolic point, after an small perturbation.

However, since we want to show the relation of these objects with heterodimensional cycles, we will give a proof based on some results by Ma\~n\'e and Gan-Wen and we will explain why the arguments of \cite{ABCDW} still are valid in the volume preserving context.

First of all, we will recall some good properties of the periodic set that will be very useful. We say that an $f$-invariant compact set $\Lambda$ has a {\it dominated splitting} if there exist a continuous splitting $T_{\Lambda}M=E\oplus F$ and constants $m\in\N$, $0<\lambda<1$ such that for every $x\in \Lambda$ we have:
$$
\|Df_x^m| E(x)\|\; \|Df_{f^m(x)}^{-m}| F(f^m(x))\|\leq \lambda.
$$

Now, let $\Lambda_i(f)$ denote the close of the set formed by hyperbolic periodic points of $f$ with index $i$.
The following proposition which is the volume preserving version of a result by Ma\~n\'e, proposition II.1 of \cite{Mane} (see also the work of Liao \cite{L}) is essential. It can be deduced from Franks Lemma \ref{l.franks} and adapting some arguments of Ma\~n\'e. We will give the necessarily adaptations in an appendix.

\begin{proposition}
\label{p.domina}
If $f\in \gm$, there exist a neighborhood $\mathcal{U}$ of $f$ in $\dm$, and constants  $K>0$, $m\in\N$ and $0<\lambda<1$ such that
\begin{itemize}
\item[a)] For every $g\in\mathcal{U}$ and $p\in Per(g)$ with minimum period $\tau(p)\geq m$
$$
\prod_{i=0}^{k-1} \|Dg^m(g^{mi}(p))| E^s_g(g^{mi}(p))\|\leq K\lambda^{k}$$ and $$ \prod_{i=0}^{k-1} \|Dg^{-m}(g^{-mi}(p))| E^u_g(g^{-mi}(p))\|\leq K\lambda^{k},
$$
where $k=[\tau(p)/m]$.

\item[b)] For all $0<i<dim M$ there exists a continuous splitting $T_{\Lambda_i(g)}M=E_i\oplus F_i$ such that
$$
\|Dg^m(x)| E_i(x)\|\; \|Dg^{-m}(g^m(x))| F_i(g^m(x))\|\leq \lambda.
$$
for all $x\in \Lambda_i(g)$ and $E_i(p)=E^s_g(p)$, $F_i(p)=E^u_{g}(p)$ when $p\in Per(g)$ and $dim E^s_{g}(p)=i$.

\item[c)] For all $p\in Per(g)$
$$
\limsup_{n\rightarrow +\infty} \frac{1}{n}\sum_{i=0}^{n-1}\log \|Dg^m(g^{mi}(p))| E^s_g(g^{mi}(p))\|<0
$$
and
$$
\limsup_{n\rightarrow +\infty} \frac{1}{n}\sum_{i=0}^{n-1}\log \|Dg^{-m}(g^{-mi}(p))| E^u_g(g^{-mi}(p))\|<0.
$$
\end{itemize}
\label{DS prop}\end{proposition}

\begin{proof}[Proof of Proposition \ref{p.index}]

If proposition \ref{p.index} is not true, then there exist two hyperbolic periodic orbits $p$ and $q$ of $f$ with respectively indices $i$ and $i+j$, for some $j>0$.

Now, we state a result by Abdenur-Bonatti-Crovisier-Diaz-Wen in \cite{ABCDW} in the volume preserving case, we slightly modify the statement.

\begin{proposition}
For any neighborhood $\mathcal{U}$ of $f\in \dm$, if there exist $p,q\in Per(f)$ with indices $i$ and $i+j$, respectively, then for any positive integer $\alpha$ between $i$ and $i+j$ there exist $g\in\mathcal{U}$ and a hyperbolic periodic point of $g$ with index $\alpha$.
\label{propABCDW}\end{proposition}

We will explain why this proposition holds in the volume preserving case later. Using this proposition, we can find hyperbolic periodic points $p$ and $q$ of $f$, by some perturbation of $f$, with indices $i$ and $i+1$, respectively.

In the sequence, we will see how to perturb $f$ in order to get a heterodimensional cycle between these two hyperbolic periodic points.
First, we remember a result by Bonatti-Crovisier \cite{BONATTICROVISIER}:

\begin{lemma}
There exists a residual subset $\SR$ of $\dm$ such that if $g\in \SR$ then $M=H(p,g)$, where $H(p,g)$ is the homoclinic class for a hyperbolic periodic point $p$ of $g$. In particular, $g$ is transitive.
\label{BC}\end{lemma}

Hence, perturbing and using the hyperbolicity of $p$ and $q$, we can suppose that our diffeomorphism $f\in \SR$ and so it is transitive. Then using the connecting lemma of Hayashi for conservative diffeomorphisms, see \cite{BONATTICROVISIER}, we can create an intersection between $W^u(p)$ and $W^s(q)$, also perturbing if necessary, we can assume that this intersection is transversal. Hence, this intersection is robust, and we can suppose that this new diffeomorphism also belongs to $\SR$. Using the connecting lemma once more, we can create an intersection between $W^s(p)$ and $W^u(q)$. Thus we create a heterodimensional cycle. We observe that this type of argument appears in \cite{Ab}.

As we comment above, what we want to do now is to find a periodic point with at least one Lyapunov exponent close to zero. We can suppose that $p$ and $q$ are fixed points.

Now, let $\mathcal{R}_1$ be the set of volume preserving diffeomorphisms where the homoclinic class are disjoint or coincide. Using a result by Carballo, Morales and Pacifico in \cite{CMP} we know that $\mathcal{R}_1$ is a residual subset in $\dm$.  So, we can assume that $f\in\SR\cap \mathcal{R}_1$. Hence, $M=\Lambda_i(f)=\Lambda_{i+1}(f)$, i.e., hyperbolic periodic points with indices $i$ and $i+1$ are dense in $M$. Therefore, using Proposition \ref{DS prop} we have a dominated splitting for $f$, $TM=E\oplus C\oplus F$, such that the dimension of $E$ and $C$ are equal to $i$ and one, respectively. And thus, by the continuation of the dominated splitting we still have this one for the perturbation of $f$ that exhibits a heterodimensional cycle between $p$ and $q$.

Let $\mathcal{U}$ be some neighborhood of $f$ in $\gm$ and $\mathcal{U}_1\subset \mathcal{U}$ such that we still have the previous dominated splitting
for every $g\in \mathcal{U}_1$.

We recall now a perturbation result of Xia in \cite{Xia}.
\begin{lemma}
Fix $\phi\in \dm$, there exist constants $\eps_0>0$  and $c > 0$, depending only of
$\phi$, such that for any $x\in M$, any $\psi\in \dm$ $\eps_0-C^1$ close to $\phi$
and any positive numbers $0<\delta<\eps_0$ and $0<\eps<\eps_0$, we have that if $d(y,x) < c\delta\eps$, then there is a $\psi_1\in \dm$ $\eps-C^1$ close to $\psi$  such that
$\psi_1(\psi^{-1}(x))=y$ and $\psi_1(z)=z$ for all $z\not\in \psi^{-1}(B_{\delta}(x))$.
\label{xia lemma}\end{lemma}

Then, we fix $\eps_0$ and $c>0$ for $f$ according the previous lemma. Now, let $0<\eps<\eps_0$ be such that if $f_1\in \dm$ is $\eps-C^1$ close to $f$ then $f\in \mathcal{U}_1$. Let $x\in W^s(p)\cap W^u(q)$ and $y\in W^u(p)\cap W^s(q)$.

Now, let $B_{p}$ and $B_{q}$ be small balls in $M$ centered at $p$ and $q$, respectively. Moreover, given any $\gamma>0$ we may choose $B_{p}$ such that $\|Df(z)-Df(p)\|\leq \gamma$, for every $z\in B_{p}$. By the choice of $x$ and $y$ we can choose $m_1, m_2, m_3$ and $m_4$ positive integers such that $f^{m_1}(x), f^{-m_3}(y)\in B_{p}$ and $f^{-m_2}(x), f^{m_4}(y)\in B_{q}$. Now, let $0<\delta<\eps_0$ such that
$$
f^{-1}(B_{\delta}(f^{m_1}(x)))\cap B_{\delta}(f^{m_1}(x))=\varnothing,\text{  } f^{-1}(B_{\delta}(f^{-m_2+1}(x)))\cap B_{\delta}(f^{-m_2+1}(x))=\varnothing,
$$
and
$$
f^{-1}(B_{\delta}(f^{m_4}(y)))\cap B_{\delta}(f^{m_4}(y))=\varnothing,\text{  } f^{-1}(B_{\delta}(f^{-m_3+1}(y)))\cap B_{\delta}(f^{-m_3+1}(y))=\varnothing.
$$

Using the $\lambda$-Lemma, we can find $z_m$ $c\delta\eps-$close to $f^{m_1}(x)$ such that $f^m(z_m)$ is also $c\delta\eps-$close to $f^{-m_3}(y)$ and $f^{r}(z_m)\in B_{p}$, $0\leq r\leq m$, for every $m>0$ large enough. 
Analogously, we can find $\overline{z}_n$ satisfying similar conditions exchanging $p$ for $q$ and the respectively iterates of $x$ and $y$. 

Hence, the set $$O_{mn}=\{z_m,..., f^{m}(z_m),f^{-m_3}(y),..., f^{m_4}(y), \overline{z}_n,..., f^{n}(\overline{z}_n), f^{-m_2}(x),..., f^{m_1}(x)\}$$
is a pseudo periodic orbit. Using lemma \ref{xia lemma}, we can perturb $f$ in order to find a periodic orbit $p_{mn}$ 
that shadows $O_{mn}$. Moreover, note that $\{z_m,..., f^{m-1}(z_m),\newline f^{-m_3}(y),..., f^{m_4-1}(y), \overline{z}_n,..., f^{n-1}(\overline{z}_n), f^{-m_2}(x),..., f^{m_1-1}(x)\}$ is the orbit of the periodic point $p_{mn}$. Moreover, $p_{mn}$ pass $m$ and $n$ times in $B_{p}$ and $B_{q}$, respectively. Furthermore, using the dominated splitting of $f$, we observe that $m$ and $n$ could be chosen such that the index of $p_{mn}$ is either $i+1$ or $i$.


Now, fix some large $n$ and choose $m=m(n)$ as the biggest one such that, $p_{mn}$ and $p_{m-1\,n}$ are hyperbolic periodic points of different perturbations of $f$ with indices $i$ and  $i+1$, respectively. We will call these perturbations of $f$ by $g$ and $h$, i.e., $p_{mn}\in Per(g)$ and $p_{m-1\,n}\in Per(h)$. We would like to remark that the way to perturb $f$ in order to construct these points gives us $g=h$ outside $B_p$. Finally, taking $n$ large enough if necessary we have that $g,h\in \mathcal{U}_1$.


By the previous process we have that the orbit of the hyperbolic periodic points $p_{k n}$, $k=m,\, m-1$,  is
$$\{z_k,..., f^{k-1}(z_k),f^{-m_3}(y),..., f^{m_4-1}(y), \overline{z}_n,..., f^{n-1}(\overline{z}_n), f^{-m_2}(x),..., f^{m_1-1}(x)\},$$
 where $z_k$ and $\overline{z}_n$ can be found by $\lambda-$lemma, depending of $k$, as before. 

Now, denoting by $\tau$ the period of $p_{mn}=f^{-m_3}(y)$ and taking $K=\|Df(p)|C(f)\|$, we have
\begin{align}
0&<\frac{1}{\tau}\log\|Dg^{\tau}(p_{mn})|_{C(g)}\|
=\frac{1}{\tau}\sum_{t=0}^{\tau-1}\log\|Dg(g^t(p_{mn}))|_{C(g)}\|
\nonumber\\
&<\frac{1}{\tau}(\sum_{t=0}^{\tau-1}\log\|Df(g^t(p_{mn}))|_{C(f)}\|+\gamma)
\nonumber\\
&\leq\frac{1}{\tau}(\sum_{t=0}^{\tau-m-1}\log\|Df(g^t(p_{mn}))|_{C(f)}\|+m(\log\|Df(p)|_{C(f)}\|+\gamma)+\gamma\tau)
\nonumber\\
&<\frac{1}{\tau-1}(\sum_{t=0}^{\tau-m-1}\log\|Df(g^t(p_{mn}))|_{C(f)}\|+(m-1)(\log\|Df(p)|_{C(f)}\|))+2\gamma+\frac{K}{\tau},
\label{equ6}\end{align}
where we use that the central direction $C$ is one-dimensional in the first equality, the continuity of the dominated splitting in the second line and the choice of $B_{p}$ in the third one.
On the other hand, using the hyperbolic periodic point $p_{m-1\,n}=f^{-m_3}(y)$ of $h$ and similarly arguments we have the following:
\begin{align}
0&>\frac{1}{\tau-1}\log\|Dh^{\tau-1}(p_{m-1\,n})|_{C(h)}\|
=\frac{1}{\tau-1}\sum_{t=0}^{\tau-1}\log\|Dh(h^t(p_{m-1\, n})|_{C(h)}\|
\nonumber\\
&>\frac{1}{\tau-1}(\sum_{t=0}^{\tau-1}\log\|Df(h^t(p_{m-1\, n}))|_{C(f)}\|-\gamma)
\nonumber\\
&\geq\frac{1}{\tau-1}(\sum_{t=0}^{\tau-m-1}\log\|Df(h^t(p_{m-1\, n}))|_{C(f)}\|+(m-1)(\log\|Df(p)|_{C(f)}\|))-2\gamma.
\label{eq7}\end{align}
Now, since $g=h$ outside $B_p$ and the orbit of $p_{mn}$ and $p_{m-1\, n}$ also coincides outside $B_p$ we have $g^t(p_{mn})=h^t(p_{m-1\,n})$ for $0\leq t\leq \tau-m-1$. Hence, we can replace the inequality (\ref{eq7}) in (\ref{equ6}), to obtain the following

\begin{align}
0&<\frac{1}{\tau}\log\|Dg^{\tau}(p_{mn})|C(g)\|< 4\gamma +\frac{K}{\tau}.
\end{align}

Therefore, since the period $\tau$ goes to infinity when $n$ goes to infinity, and $\gamma>0$ is arbitrary,  it is possible to find a hyperbolic periodic point $p_{mn}$  with a Lyapunov exponent sufficiently close to zero.

In the general case, when $p$ and $q$ are hyperbolic periodic points, the difference is that the neighborhoods $B_p$ and $B_q$ must be neighborhoods of the orbits of $p$ and $q$, respectively, and then the numbers $m$ and $n$ will be multiples of the periods of $p$ and $q$, respectively. Hence, by the same arguments as before we can find the periodic point $p_{mn}$ with at least one Lyapunov exponent sufficiently close to zero.

Finally, using Franks lemma \ref{l.franks} again, we can perturb it once more such that we reduce a little bit the force of the eigenvalue associated to the Lyapunov exponent close to zero in each point of the orbit of $p_{mn}$ such that we can get indeed a zero Lyapunov exponent for the periodic point $p_{mn}$. This means, we have an eigenvalue with absolute value one, and then $p_{mn}$ is not a hyperbolic periodic point. Since all of these perturbations can be done inside $\mathcal{U}\subset \mathcal{F}_m^1(M)$ we have a contradiction.
\end{proof}


\begin{remark}
We observe that we could do the same arguments, using only that there are periodic points with different indices, as it is done in \cite{GW}. But, the proof is slightly simplified after that we found periodic points with indices $i$ and $i+1$.
\end{remark}

Now, we explain how proposition \ref{propABCDW} could be proved in the volume pre\-ser\-ving case with the same arguments in the dissipative case. First, we will prove a conservative version of the Proposition 2.3 in \cite{ABCDW}.

\begin{proposition}
There is a residual subset $\SR_2$ of $\dm$ consisting of diffeomorphisms $f$ such that $Per_{\R}(H(p,f))$, the set of hyperbolic periodic points of $f$ with the same index of $p$ and with only real eigenvalues of multiplicity one, is dense in $H(p,f)$ for every non-trivial homoclinic class $H(p,f)$ of $f$.
\label{proABCDW}\end{proposition}

We recall that a {\it periodic linear systems (cocycles)} is a 4-tuple $\mathcal{P}=(\Sigma, f, \mathcal{E},A)$, where $f$ is a diffeomorphism, $\Sigma$ is an infinite set of periodic points of $f$, $\mathcal{E}$ an Euclidian vector bundle defined over $\Sigma$, and $A\in GL(\Sigma, f,\mathcal{E})$ is such that $A(x):\mathcal{E}_x\rightarrow \mathcal{E}_{f(x)}$ is a linear isomorphism for each $x$ ($\mathcal{E}_x$ is the fiber of $\mathcal{E}$ at $x$). For the precise definition we refer the reader to the work of  Bonatti-Diaz-Pujals \cite{BDP}.

\begin{lemma}[Lemma 1.9 in \cite{BDP}] Let $H(p,f)$ be a non-trivial homoclinic class. Then the derivative $Df$ of $f$ induces a periodic linear system with transitions over $Per_h(H(p,f))$, the set of hyperbolic periodic points homoclinically related with $p$.
\end{lemma}

Using this lemma and Franks lemma \ref{l.franks}, our problem in Proposition \ref{proABCDW} becomes a problem of linear algebra. We say that a periodic linear system with transitions $\mathcal{P}=(\Sigma,\, f,\,\mathcal{E},\, A)$ is {\it diagonalizable} at the point $x\in \Sigma$ if the linear map
$$
M_A(x):\mathcal{E}_x\rightarrow \mathcal{E}_x,\quad M_A(x)=A(f^{\tau(x)-1}(x))\circ\ldots\circ A(f^2(x))\circ A(x),
$$
only has positive real eigenvalues of multiplicity one.

\begin{lemma} For every periodic linear system with transitions $\mathcal{P}=(\Sigma,\, f,\,\mathcal{E},\, A)$ and every $\eps>0$ there is a dense subset $\Sigma'$ of $\Sigma$ and an $\eps-$perturbation $A'$ of $A$ defined on $\Sigma'$ which is diagonalizable, that is, $M_{A'}(x)$ has positive real eigenvalues of multiplicity one for every $x\in \Sigma'$.
\end{lemma}

By Remark 7.2 in \cite{BDP} we can consider the perturbation $A'$ such that $det A'(x)=1$ for every $x\in\Sigma'$.  Then, as we have noted before, we can use Franks Lemma \ref{l.franks} and the previous lemmas to show  proposition \ref{proABCDW}. In fact, after we have done all these observations the proof is identically the proof of proposition 2.3 in \cite{ABCDW}.

Hence, we can suppose $f\in \mathcal{R}_2\cap \mathcal{R}$, where the residual set $\SR$ is given by Proposition \ref{BC}, and then we may assume that $p$ and $q$ have only real eigenvalues of multiplicity one for $Df^{\tau(p)}(p)$ and $Df^{\tau(q)}(q)$, respectively.

As we did before we can suppose, unless some perturbation, that $f$ exhibits a heterodimensional cycle between $p$ and $q$. Hence the proof of the Proposition \ref{propABCDW} follows direct from the next result stated in \cite{ABCDW}.

\begin{proposition}[Theorem 3.2 in \cite{ABCDW}] Let $f$ be a diffeomorphism having a heterodimensional cycle associated to periodic saddles $p$ and $q$, of indices $i$ and $i+j$ with $j>0$, with real eigenvalues. Then, for any $C^1-$neighborhood $\mathcal{U}$ of $f$ and for any integer $\alpha$ with $i\leq \alpha \leq i+j$, there exists $g\in \mathcal{U}$ having a periodic point with index $\alpha$.
\end{proposition}

Although this proposition has been stated there for dissipative diffeomorphisms, all of the used perturbations of $f$ are applications of Franks lemma and Hayashi's connecting lemma. Then, this proposition is still true in the volume preserving case since these two perturbation tools are available in the conservative setting.

\section{Proof of Theorem A}

In the sequence we shall prove the hyperbolicity of $\ov{Per(f)}$ for every $f\in\gm$, since we already know that the indices of hyperbolic periodic points is  constant, say $s$.

Let us fixe $f\in \gm$ and a continuous dominated splitting $T_{\ov{Per(f)}}M=E\oplus F$ given by the Proposition \ref{DS prop}. From now, we also consider $m\in\N$, $0<\lambda<1$ and $K>0$ as in the Proposition \ref{DS prop}. We will prove that this splitting is hyperbolic. We will follow here similar arguments as in the proof of Mane of Theorem B in \cite{Mane}.

To show this we need prove that we have contraction and expansion in the sub-bundles $E$ and $F$, respectively, unless a certain finite time iterate. Hence, by compactness of $\ov{Per(f)}$, we just need to show the following
\begin{equation}
\liminf_{n\rightarrow +\infty} \|Df^n(x)|E(x)\|=0
\label{eq 1}\end{equation}
and
$$
\liminf_{n\rightarrow +\infty} \|Df^{-n}(x)|F(x)\|=0,
$$
for all $x\in \ov{Per(f)}$.

Observe it's enough to prove the first case since the second one can be deduced from the first one replacing $f$ by $f^{-1}$.

Suppose now $(\ref{eq 1})$ is not true. Then there exists $x\in M$ such that
$$
\|Df^{jm}(x)|F(x)\|\geq c>0, \text{ for all } j>0.
$$
Defining the following probability measure $$\mu_j=\frac{1}{j}\sum_{i=0}^{j-1}\delta_{f^{mi}(x)},$$
where $\delta$ is the dirac measure, we can find a subsequence $j_n\rightarrow \infty$ such that $\mu_{j_n}$ converges  to an $f^m-$invariant probability measure $\mu$ in the weak$^*$ topology and
\begin{equation}
\lim_{n\rightarrow +\infty} \frac{1}{j_n}\log \|Df^{mj_n}(x)|E(x)\|\geq 0.
\end{equation}

Hence, taking the continuous functional  $\phi(x)=\log\|Df^{m}(x)|E(x)\|$ over $\ov{Per(f)}$, we obtain:
\begin{align*}
\int_{\ov{Per(f)}}\phi \; d\mu &= \lim_{n\rightarrow +\infty}\frac{1}{j_n}\sum_{i=0}^{j_n-1}\log\|Df^{m}(f^{mi}(x))|E(f^{mi}(x))\|
\\
&\geq \lim_{n\rightarrow +\infty}\frac{1}{j_n}\log\|Df^{mj_n}(x)|E(x)\|\geq 0.
\end{align*}

And so, using Ergodic Birkhoff's Theorem
\begin{equation}
0\leq \int_{\ov{Per(f)}}\phi\;d\mu=\int_{\ov{Per(f)}}\lim_{n\rightarrow +\infty}\frac{1}{n}\sum_{i=0}^{n-1}\log\|Df^{m}(f^{mi}(x))|E(f^{mi}(x))\|\; d\mu.
\label{eq 2}\end{equation}

Now let $\Sigma(f)\subset M$ a total probability set given by the Ergodic closing lemma in the volume preserving case, see \cite{Arnaud}. Hence, denote by $\nu=\frac{1}{m}\sum_{i=0}^{m-1} {f^i}^*\mu$ the $f-$invariant probability measure induced by $\mu$, we have $\nu(\Sigma(f)\cap \ov{Per(f)})=1$ since $\nu$ is supported on $\ov{Per(f)}$. But now, by the invariance of $\Sigma(f)\cap \ov{Per(f)}$ for $f$, it's easy to see that this is also a total probability set for $\mu$. And so, this together with $(\ref{eq 2})$ imply  the existence of a point $y\in\Sigma(f)\cap\ov{Per(f)}$ such that:
\begin{equation}
\lim_{n\rightarrow +\infty}\frac{1}{n}\sum_{i=0}^{n-1}\log\|Df^{m}(f^{mi}(y))|E(f^{mi}(y))\|\geq 0.
\label{eq 3}\end{equation}

Observe that part (c) of Proposition \ref{DS prop} is an obstruction for $y$ be periodic. Hence $y\not\in Per(f)$.

By (\ref{eq 3}), we can take $\lambda<\lambda_0<1$ and $n_0>0$ such that:
\begin{equation}
\frac{1}{n}\sum_{i=0}^{n-1}\log\|Df^{m}(f^{mi}(y))|E(f^{mi}(y))\|\geq \log \lambda_0,
\label{eq 4}\end{equation}
when $n\geq n_0$.

In the next step we will find a hyperbolic periodic point $p\in Per(g)$ such that its orbit is ``close" to the orbit of $y$, for $g$ near $f$, and then we will use Lemma \ref{l.franks} to exchange the derivative at the orbit of $p$, such that the inequality (\ref{eq 4}) gives us a contradiction with part $(a)$ of proposition \ref{DS prop}.

Using that $y\in \Sigma(f)$ we can approximate $f$ by diffeomorphisms $g$ such that there exists $p\in Per(g)$ and the distance between $f^j(p)$ and $f^{j}(y)$ is arbitrary small, for $0\leq j\leq n$, where $n$ is the minimum period of $p_g$. Since $y$ is not periodic the period $n$ must goes to infinity when $g$ approaches $f$. Hence we may choose $g$ and $p$ such that:
\begin{equation}
n\geq m,
\end{equation}
\begin{equation}
k\geq n_0,
\end{equation}
\begin{equation}
K\lambda^k<\lambda_0^{k}
\end{equation}
and
\begin{equation}
\left(\frac{\lambda}{\lambda_0}\right)^k C^m\leq \frac{1}{2},
\end{equation}
where $k=[n/m]$ and $C=\sup_{x\in M}\|Df^{-1}(x)\|$.

These choices together with $(\ref{eq 4})$ and the dominated splitting of $f|\ov{Per(f)}$ give us the following
\begin{align}
\|Df^{-n}_{f^n(y)}|F(f^n(y))\| &\leq \prod_{i=0}^{k-1}\|Df^{-m}_{f^{n-mi}(y)}|F(f^{n-mi}(y))\|\; \|Df^{-(n-mk)}_{f^{n-mk}(y)}|F(f^{n-mk}(y))\|
\nonumber\\
&\leq \lambda^k\, C^m\, \lambda_0^{-k}\leq \frac{1}{2}.
\label{eq 5}\end{align}

Let $U$ be a neighborhood of $\ov{Per(f)}$ small enough such that the maximal set in $U$ $$\Lambda_U(f)=\bigcap_{n\in\Z} f^n(U)$$  has a dominated splitting and satisfying the thesis of the Proposition \ref{DS prop}.  Hence, we can chose $\mathcal{U}\subset \gm$ a neighborhood of $f$ such that every $h\in\mathcal{U}$ has a dominated splitting in $\Lambda_U(h)$  near of the dominated splitting in $\Lambda_U(f)$.
Observe that $E_g(p)=E^s_g(p)$ and $F_g(p)=E^u_g(p)$ since dominated splitting is unique if we fix the dimensions, and the index of periodic points is constant for $g\in\mathcal{U}$. Hence, taking a smaller neighborhood $\mathcal{U}$ if necessary we can suppose $E_g^s(g^i(p))$ and $E^u_g(g^i(p))$ as near as we want of $E(f^i(y))$ and $F(f^i(y))$, $0\leq i \leq n$, respectively.

In the sequence we will build some volume preserving isomorphisms $A_i:T_{g^i(p)}M\rightarrow T_{f^{i}(y)}M$, $0\leq i\leq n$, near of identity in local coordinates. Moreover, for future convenience, we will have $A_i(E^s_g(g^i(p)))=E(f^i(y))$ isometrically and $A_i(E^u_g(g^i(p)))=F(f^i(y))$, $0\leq i\leq n$.

We show how to construct $A_0$, the other cases are analogous. We choose
$$
\{e_1(i),\,\ldots,\,e_s(i),\, r_1(i),\,\ldots,\,r_{d-s}(i)\} \,\text{ a basis of } T_i M,\, i=y,p,
$$
such that $\{e_j(i),\, 1\leq j\leq s\}$ is an orthonormal basis for $E(y)$ if $i=y$ or for $E^s_g(p)$ if $i=p$, and $\{r_j(i),\, 1\leq j\leq d-s\}$ is an orthonormal basis for $F(y)$ if $i=y$ or for $E^u_g(p)$ if $i=p$.

Let $A_0: T_p M\rightarrow T_yM$ be a linear map satisfying $A_0(e_j(p))=e_j(y)$ and $A_0(r_j(p))=r_j(y)$, $1\leq j\leq d$. Therefore, by construction, $A_0$ is a volume preserving linear map.

Now, let us back to the proof. Let $L_i:T_{g^i(p)} M\rightarrow T_{g^{i+1}(p)} M$ be volume preserving maps defined as follows
$$
L_i= A^{-1}_{i+1}\, Df(f^i(y))\, A_i, \, \text{ for } 0\leq i\leq n-1.
$$

Hence, taking $n$ large enough if necessary, $L_i$ is as close of $Dg(g^i(p))$ as we want, for all $0\leq i\leq n-1$. Then, using lemma \ref{l.franks}, we can find $h\in \mathcal{U}$ such that $p\in Per(h)$ and $Dh(h^i(p))=L_i$, $0\leq i\leq d-1$. Observe that $E^s_g(p)$ and $E^u_g(p)$ still are invariants by $Dh^n(p)$, by construction of $L_i'$s. This together with (\ref{eq 5}), the proximity of $f$ and $g$, and the dimension of the subspaces give us that $E^u_h(p)=E^u_g(p)$. And so, $E^s_h(p)=E^s_g(p)$ too.

Finally, since $A_i|E^s_g(g^i(p))$ is an isometry, we have the following
$$
\|Dh^{m}(h^{im}(p))| E^s_h(h^{im}(p))\|=\|Df^{m}(f^{im}(y))| E(f^{im}(y))\|, \, \text{ for all }\, i\in\N.
$$
Therefore,
$$
\prod^{k-1}_{i=0} \|Dh^{m}(h^{im}(p))| E^s_h(h^{im}(p))\|= \prod^{k-1}_{i=0}\|Df^{m}(f^{im}(y))| E(f^{im}(y))\|\geq \lambda_{0}^{k},
$$
what contradicts Proposition \ref{DS prop}. Therefore we showed that $\ov{Per(f)}$ is hyperbolic if $f\in \gm$.


Finally to conclude that if $f\in \gm$ then $f$ is Anosov, we just need to show that $\Omega(f)=\ov{Per(f)}$ since $\Omega(f)=M$, by  Poincar\'e's Recurrence Theorem. This will be a consequence of Pugh's closing lemma.

If $f\in\gm$ then there exists some neighborhood $\mathcal{U}$ of $f$ in $\dm$ such that $\sharp H_n(g)$, the number of hyperbolic periodic points of $g$ with period smaller or equal than $n$, is finite and equal for every $g\in\mathcal{U}$, since all diffeomorphisms in $\mathcal{U}$ has only hyperbolic periodic points.

Now, suppose $\ov{Per(f)}\subsetneq \Omega(f)$, and let $x\in\Omega(f)\backslash \ov{Per(f)}$. By Pugh's closing lemma we can fix $k\in \N$ such that all of the perturbations of $f$, needed to create a hyperbolic periodic point near of $x$, are done in an arbitrary small neighborhood of $$\bigcup_{-k\leq j\leq k} f^j(x).$$

Thus, let $U$ be a neighborhood of $\ov{Per(f)}$ such that $ f^j(x)\not \in \overline{U}$, $-k\leq j\leq k$. So, using the closing lemma we can get $g\in \mathcal{U}$ and $p\in Per(g)$ with $p\not\in \overline{U}$. However by choice of $k$ and $U$, $f$ is equal to $g$ in $U$ and since $p\not \in \overline{U}$ we have $H_n(f)\neq H_n(g)$ for some $n\in \N$, what contradicts the fact of $g\in\mathcal{U}$. Therefore, we have $\Omega(f)=\ov{Per(f)}$ and this completes the proof.

\section{Palis Conjecture in the volume preserving scenario}

In this section, we elaborate the arguments due to Crovisier in \cite{CROVISIER} on the Palis conjecture in the volume preserving scenario, and this gives the proof of corollary \ref{palis conj}. As we said in the introduction we only need to prove that if the dimension of $M$ is greater than two.

Suppose $f\in\dm$ is not Anosov then $f\not\in \gm$. Therefore by theorem A, there exists $g\in \dm$ close to $f$ with a non hyperbolic periodic point $p$. Thus, using Franks lemma \ref{l.franks}, we can bifurcate this periodic point and produce two hyperbolic periodic points $q$ and $r$ with indices $i$ and $i+1$ respectively.

Thus, as we did before in the proof of Proposition \ref{p.index}, we can perturb once more and create a heterodimensional cycle between $q$ and $r$.

For more details we refer the reader to \cite{Crovisier2}.

\section{Star Flows}

In this section we observe how we can extend some results in the theory of star flows to a divergence free context. Let $\vf$ be the set of vector fields on $M$ endowed with the $C^1$ topology. And $\vm\subset \vf$ the subspace of vector fields which are divergence free. By Liouville's formula, we know that the flow generated by a divergence free vector field is volume preserving, so they are also called incompressible flows in the literature.

The analogous version of $\gm$ to the flow case are called incompressible star flows. We say that $X$ generates an incompressible star flow if there exits a neighborhood $\SU$ of $X$ in $\vm$ such that if $Y\in \vm$ then all of its singularities and periodic orbits are hyperbolic. We denote the set of incompressible star flows by $\sm$.

The analogous result by Gan and Wen in \cite{GW} should be true for incompressible star flows.

\begin{theorem}
If $X\in\sm$ has no singularities then $X$ is Anosov.
\end{theorem}

In fact, this could be seen as follows. First of all, proposition \ref{p.domina} can be extended to the context of incompressible flows exactly as Hayashi did in \cite{Hayashi}. And the same calculations that we did on heterodimensional cycles (which are inspired on the calculations of \cite{GW}) can be done to prove the following proposition.

\begin{proposition}
If $X\in \sm$ has no singularities then $X$ has no heterodimensional cycles.
\end{proposition}

Now, we observe that the version of Bonatti-Crovisier's result lemma \ref{BC} to incompressible flows can be found in \cite{Bessa}. Hence, together with the previous proposition we obtain the following proposition.

\begin{proposition}
If $X\in \sm$ has no singularities then there exist a neighborhood $\SU$ of $X$ in $\vm$ and $i\in \N$, such that the index of any periodic orbit of any vector field $Y\in \SU$ is $i$.
\end{proposition}

Now, we can adapt the arguments of Toyoshiba \cite{Toyo}, in the same way that as we did before, to obtain the following result.

\begin{proposition}
If $X\in \sm$ has no singularities and there exist a neighborhood $\SU$ of $X$ in $\vm$ and $i\in \N$, such that the index of any periodic orbit of any vector field $Y\in \SU$ is $i$. Then $X|_{\ov{Per(X)}}$ is hyperbolic and $\ov{Per(X)}=\Omega(X)$.
\end{proposition}

Finally, by Poincar\'e's recurrence, we know that $\Omega(X)=M$ and this would imply that $X$ is Anosov.

If we denote by $KS$ the $C^1$-interior of Kupka-Smale incompressible vector fields then $KS\subset \sm$. Now, suppose that $X\in KS$ and $X$ has a singularity $\sigma$. Then since $\Omega(X)=M$ and there are a finite number of singularities then $\sigma$ is approximated by regular orbits and so, by the connecting lemma, after a perturbation we obtain that $W^u(\sigma)\cap W^s(\sigma)-\{\sigma\}\neq \emptyset$. Thus this would be a non-transversal intersection, but this is a contradiction since $X\in KS$. In particular, the analogous result of Toyoshiba holds for incompressible flows.

\begin{corollary}
Let $X\in KS$ then $X$ is Anosov.
\end{corollary}

Finally, we observe that if $X$ is a $C^1$-structurally stable divergence free vector field then $X\in KS$. So the following corollary generalizes a corollary found in \cite{BR} to any dimension (not only 3).

\begin{corollary}
If $X\in \vm$ is a $C^1$-structurally stable divergence free vector field then $X$ is Anosov.
\end{corollary}

{\bf Acknowledgements:} A.A. wants to thank Prof. Sylvain Crovisier by his remarkable comments and for pointing out a mistake in a previous version of this article. He wants to thank Prof. Jorge Rocha for kindly provide reference \cite{BR}.  He also wants to thank ICMC-USP for the kind hospitality. T.C. wants to thank UFRJ for the kind hospitality during preparation of this work.

\section{Appendix}

In this appendix, we show the necessarily modifications to prove proposition \ref{p.domina}. In particular, we will review Ma\~n\'e's argument from \cite{Mane} from page 523 to 540.

First we recall some notions introduced by Ma\~n\'e in our context. Let $GL(d)$ be the group of linear isomorphisms and $SL(d)$ be the subgroup of $GL(d)$ of isomorphisms  with determinant equal to one. By a hyperbolic sequence we mean a sequence hyperbolic isomorphisms $\xi:\Z\to GL(d)$. The sequence is periodic if there exists $m$ such that $\xi_{j+m}=\xi_j$ for $j\in \Z$, the minimal positive $m$ is called the period of the sequence. If  the stable space is the whole $\R^d$ then we call it a contracting sequence. If the sequence is formed by volume preserving isomorphisms, i.e. $\xi:\Z\to SL(d)$ then we call it a vol-hyperbolic sequence. Also. a family of periodic sequences $\{\xi^{\alpha}\}$ will be called a vol-family.

A family of periodic sequences $\{\xi^{\alpha}\}$ is vol-hyperbolic, if all of its sequences are vol-hyperbolic and $\sup_{\alpha,j}\{\|\xi_j^{\alpha}\|\}<\infty$. Define a distance $d(\xi,\eta)=\sup_{\alpha,j}\{\|\xi_j^{\alpha}-\eta_j^{\alpha}\|$ between two families $\xi$ and $\eta$, we also say that they are equivalent if for every $\alpha$ the period of $\xi^{\alpha}$ and $\eta^{\alpha}$ coincide.
Moreover the family $\xi$ is uniformly vol-hyperbolic if there exists $\eps>0$ such that every equivalent vol-family $\eta$ which satisfies $d(\xi,\eta)<\eps$ is vol-hyperbolic.

The proof of proposition \ref{p.domina} goes as the same way as in \cite{Mane}, using Franks lemma in the volume preserving case (lemma \ref{l.franks}) and the following proposition, which is analogous to lemma II.3 of \cite{Mane}.

\begin{proposition}
\label{l.linear}
If $\{\xi^{\alpha}\}$ is an uniform vol-hyperbolic family then there exist constants $K>0$, $m\in \N$ and $0<\lambda<1$ such that
\begin{enumerate}
\item If $\xi^{\alpha}$ has period $n\geq m$ and $k=[n/m]$ then
$$\prod_{j=0}^{k-1}\|(\prod_{i=0}^{m-1}\xi^{\alpha}_{mj+i})|_{E^s_{mj}}\|\leq K\lambda^k\textrm{ and }\prod_{j=0}^{k-1}\|(\prod_{i=0}^{m-1}\xi^{\alpha}_{mj+i})^{-1}|_{E^u_{m(j+1)}}\|\leq K\lambda^k.$$
\item For every $\alpha$ and $j\in \Z$:
$$\|(\prod_{i=0}^{m-1}\xi^{\alpha}_{j+i})|_{E^s_{j}}\|\|(\prod_{i=0}^{m-1}\xi^{\alpha}_{j+i})^{-1}|_{E^u_{j+m}}\|\leq \lambda.$$
\item For every $\alpha$:
$$\limsup_{n\to\infty}\frac{1}{n}\sum_{j=0}^{n-1}\log\|(\prod_{i=0}^{m-1}\xi^{\alpha}_{mj+i})|_{E^s_{mj}}\|<0\textrm{ and}$$
$$\limsup_{n\to\infty}\frac{1}{n}\sum_{j=0}^{n-1}\log\|(\prod_{i=0}^{m-1}\xi^{\alpha}_{mj+i})^{-1}|_{E^u_{m(j+1)}}\|<0.$$
\end{enumerate}
\end{proposition}

We only indicate the necessarily adaptations.

First, we recall lemma II.7 from \cite{Mane}, that will be used to control the stable and unstable parts of the isomorphisms.

\begin{lemma}
Let $\{\xi^{\alpha}\}$ be a uniformly contracting family. There exists $K>0$, $m\in \N$ and $0<\lambda<1$ such that
\begin{enumerate}
\item If $\xi^{\alpha}$ has period $n\geq m$ and $k=[n/m]$ then
$$\prod_{j=0}^{k-1}\|\prod_{i=0}^{m-1}\xi^{\alpha}_{mj+i}\|\leq K\lambda^k$$.
\item For every $\alpha$:
$$\limsup_{n\to\infty}\frac{1}{n}\sum_{j=0}^{n-1}\log\|\prod_{i=0}^{m-1}\xi^{\alpha}_{mj+i}\|<0.$$
\end{enumerate}
\end{lemma}

With this we can prove the following lemma

\begin{lemma}
If $\{\xi^{\alpha}\}$ is a uniformly vol-hyperbolic family then there exist $\eps>0$, $K>0$, $0<\lambda<1$ such that if $\{\eta^{\alpha}\}$ is an equivalent vol-family with $d(\xi,\eta)<\eps$ then $\eta$ is vol-hyperbolic. Moreover, if $n$ is the period of $\eta^{\alpha}$ and $\wt{\eta}_j^{\alpha}:\R^n/E^s_j\to\R^n/E^s_{j+1}$ is the map induced by $\eta_j^{\alpha}$ then
$$\|\prod_{j=0}^{n-1}\eta_j^{\alpha}/E^s_0\|\leq K\lambda^n\textrm{ and }\|(\prod_{j=0}^{n-1}\ov{\eta}_j^{\alpha})^{-1}\|\leq K\lambda^n.$$
\end{lemma}

\begin{proof}
As Ma\~n\'e did, if $\eps>0$ is small and $0<m<d$ and we take a vol-family $\{\phi^{\beta}\}$ containing every sequence $\phi:\Z\to GL(d-m)$ with same period of some $\xi^{\alpha}$ such that $\sup_i\|\phi-(\ov{\xi}^{\alpha}_i)\|<\eps$ and this family is uniformly contracting. Indeed, if not there exist a sequence $\psi:\Z\to GL(d-n)$ such that for some $\beta$, $\psi$ and $\phi^{\beta}$ have the same period, $\sup_i\|\phi_i^{\beta}-\psi\|$ is small and $\prod_{j=0}^{n-1}\psi_j$ has an eigenvalue with modulus 1 and determinant close to 1 (since is close to $\phi$).

Then construct a sequence $\zeta$ with the same period of $\xi$ such that $\zeta_j=\psi_j^{-1}$ restricted to $\R^d/E^s_j$ and equal to $\det(\psi_j)\xi$ restricted to $E^s_j$. Note that $\zeta$ is close to $\xi$.  This would contradict the vol-uniform hyperbolicity of $\{\xi^{\alpha}\}$.

Now, we can proceed exactly as in the rest of the proof of lemma II.8 of \cite{Mane}.
\end{proof}

Thus, as in \cite{Mane}. We obtain (1) and (3) of proposition \ref{l.linear}.

Now, we prove (2) of proposition \ref{l.linear}. First, we observe that the next lemma about angles (lemma II.9 of \cite{Mane}) holds in the volume preserving case.

\begin{lemma}
If $\{\xi^{\alpha}\}$ is an uniformly vol-family then there exist $\eps>0$, $\gamma>0$ and $m\in \N$ such that if $\{\eta^{\alpha}\}$ is an equivalent vol-family with $d(\xi,\eta)<\eps$ then for every $\eta^{\alpha}$ with period $n\geq m$ the angles between stable and unstable spaces are bounded away from zero, i.e.
$$\angle (E^s_0(\eta^{\alpha}),E_0^u(\eta^{\alpha}))>\gamma.$$
\end{lemma}

Indeed, to obtain a contradiction, Ma\~n\'e construct the following perturbation $\zeta_j=\eta_j$ if $0<j\leq n-1$ and
$$\xi_0=\eta_0\left(
\begin{array}{cc}
I & C \\
0 & I\\\end{array}\right).$$
For some suitable $C$, note that this is also a vol-sequence.

Finally, if (2) does not holds then Ma\~n\'e shows that the following sequence has small angles between the stable and unstable spaces. The sequence has the form $\eta^t:\Z\to GL(d)$, for some $t$ and has period $n$ large, such that for $1\leq j<n-1$,
$$\eta^t_i=(I+T_i)\xi_i^{\alpha}(I+P_t) \textrm{ , }\eta^t_{i+j}=(I+T_{i+j})\xi_{i+j}^{\alpha}\textrm{ and }$$
$$\eta^t_{i+n-1}=(I+S_t)T_{i+n_0-1}\xi^{\alpha}_{i+n-1}.$$
However, the transformations involved satisfy the following estimate for $\eps>0$ very small,
$$\|P_t\|\leq \eps\,,\,\|S_t\|\leq \eps \textrm{ and }\|T_j\|\leq \eps.$$
Hence $\det(I+T_j)$, $\det(I+S_t)$ and $\det (I+P_t)$ are close to 1, so close as we want. Hence we can divide $\eta$ by these determinants appropriately and now we obtain a vol-sequence, with same invariant spaces. In particular the angles are small and again we obtain a contradiction with the previous lemma.

\begin{remark}
Actually, the same argument could be used to show that if a vol-family is not uniformly hyperbolic then it is not uniformly vol-hyperbolic. Indeed, it would have a not hyperbolic sequence sufficiently close. In particular, the determinant would be almost one. Thus multiplying by the inverse of the determinant in a direction distinct of the non-hyperbolic direction, we obtain a not hyperbolic vol-sequence close to the original. A contradiction.
\end{remark}

We finish this appendix with some comments about proposition \ref{p.domina} in the symplectic case. Of course if $f$ is Anosov the statements of that proposition are obviously true. In fact, such proposition is true by the arguments of Newhouse \cite{Newhouse}, since he proves that if a symplectic diffeomorphism has an homoclinic tangency than it can be approximated by one with an 1-elliptic periodic point, which is in particular a non-hyperbolic point.

However, proposition \ref{p.domina} could be obtained directly. For instance, the domination part could be obtained using arguments from \cite{HT} and \cite{Bochi}. In fact, Avila-Bochi-Wilkinson in \cite{ABW} theorem 3.5, obtains a direct proof that a partially hyperbolic non-Anosov diffeomorphism can be approximated by one with a non-hyperbolic periodic point under the hypothesis of \emph{unbreakability}. Since there are many references and ideas about this subject we will not elaborate more on it and we refer the reader to those references for more details.

\bibliographystyle{amsplain}

\end{document}